\documentclass[british, 12pt, reqno]{amsart}
\usepackage[T1]{fontenc}
\usepackage[latin9]{inputenc}
\usepackage{amsmath, amsthm, amssymb, stmaryrd}
\usepackage{mathrsfs}
\usepackage{mathtools}

\usepackage{amsmath}
\usepackage{amsthm}
\usepackage{amssymb}
\usepackage{xcolor}
\usepackage{dsfont}

\usepackage{cancel}
\usepackage{soul}
\usepackage{comment}
\usepackage{color}
\usepackage{amscd}
\usepackage{tabularx}
\usepackage{url}
\usepackage{eurosym}
\usepackage{braket}
\usepackage
{hyperref}
\usepackage{enumerate}
\usepackage{graphicx}

\makeatletter
\def\widebreve{\mathpalette\wide@breve}
\def\wide@breve#1#2{\sbox\z@{$#1#2$}%
     \mathop{\vbox{\m@th\ialign{##\crcr
\kern0.08em\brevefill#1{0.8\wd\z@}\crcr\noalign{\nointerlineskip}%
                    $\hss#1#2\hss$\crcr}}}\limits}
\def\brevefill#1#2{$\m@th\sbox\tw@{$#1($}%
  \hss\resizebox{#2}{\wd\tw@}{\rotatebox[origin=c]{90}{\upshape(}}\hss$}
\makeatletter

\makeatletter
\numberwithin{equation}{section}
\numberwithin{figure}{section}
\theoremstyle{plain}
\newtheorem{thm}{\protect\theoremname}[section]
\theoremstyle{plain}

\ifx\proof\undefined
\newenvironment{proof}[1][\protect\proofname]{\par
	\normalfont\topsep6\p@\@plus6\p@\relax
	\trivlist
	\itemindent\parindent
	\item[\hskip\labelsep\scshape #1]\ignorespaces
}{%
	\endtrivlist\@endpefalse
}
\providecommand{\proofname}{Proof}
\fi
\theoremstyle{remark}

\theoremstyle{plain}
\newtheorem{lem}[thm]{\protect\lemmaname}

\makeatother

\usepackage{babel}
\providecommand{\corollaryname}{Corollary}
\providecommand{\lemmaname}{Lemma}
\providecommand{\remarkname}{Remark}
\providecommand{\theoremname}{Theorem}

\numberwithin{equation}{section}
\numberwithin{figure}{section}
\theoremstyle{plain}

\def\1{{\textcolor{red} {1}}}
\def\d1{{\textcolor{red} {d-1}}}

\def \le {\leqslant}

\def \ge {\geqslant}

\def \d {{\mathrm{d}}}

\def \ds1 {\mathds{1}}

\def \alp {{\alpha}}

\def \lam {{\lambda}}

\makeatother

\usepackage{babel}
\providecommand{\lemmaname}{Lemma}

\providecommand{\theoremname}{Theorem}
\usepackage{xifthen,ifthen}
\makeatletter
\newcounter{@ToDo}
\newcommand{\todo@helper}[1]{%
	({\color{blue}TODO~\arabic{@ToDo}: {#1\@addpunct{.}}})%
}
\newcommand{\todo}[1]{\stepcounter{@ToDo}%
	\relax\ifmmode\text{\todo@helper{#1}}%
	\else\todo@helper{#1}\fi%
}
\newcounter{@cdo}
\newcommand{\cdo@helper}[1]{%
	({\color{red}CITE~\arabic{@cdo}: {#1\@addpunct{.}}})%
}
\newcommand{\cdo}[1]{\stepcounter{@cdo}%
	\relax\ifmmode\text{\cdo@helper{#1}}%
	\else\cdo@helper{#1}\fi%
}
\newcommand{\mmod}[1]{\,\,\mathrm{mod}\,\,#1}

\def \ds1 {\mathds{1}}

\def \alp {{\alpha}}

\def \epsilon {{\varepsilon}}

\def \lam {{\lambda}}

\begin{document}

\author{Sam Chow \and Owen Jones}

\address{Mathematics Institute, Zeeman Building, University of Warwick, Coventry CV4 7AL, United Kingdom}
\email{sam.chow@warwick.ac.uk}

\address{Mathematics Institute, Zeeman Building, University of Warwick, Coventry CV4 7AL, United Kingdom}
\email{Owen.S.Jones@warwick.ac.uk}

\title{On short averages of the Fibonacci partition function}



\begin{abstract}
We show that short averages of the Fibonacci partition function over $(H, H + H^\alp]$ can be significantly smaller than long averages over $(0,H]$, for any constant $\alp \in (0,1)$. This provides a negative answer to a question that we asked in our previous work on this topic (Journal of Number Theory, 2024).
\end{abstract}

\maketitle

\section{Introduction}

The Fibonacci partition function $R(n)$ has been investigated since the first volume of the Fibonacci quarterly in 1963, see \cite{HB1963}. For $n \ge 0$, let $R(n)$ count solutions to
\[
x_1 + \cdots + x_s = n,
\]
where $s \ge 0$ and $x_1 < x_2 < \ldots < x_s$ are Fibonacci numbers.
The values $R(n)$ form the sequence OEIS A000119. Carlitz \cite{Car1968} wrote one of the first articles on Fibonacci partitions, showing among other things that
\[
R(F_m) = \lfloor m/2 \rfloor \qquad (m \ge 2),
\]
where $F_1 = F_2 = 1$ and $F_{m+2} = F_{m+1} + F_m$ define the Fibonacci sequence.

We know from the existence of Zeckendorf expansions \cite{Zec1972} that 
\[
R(n) \ge 1
\]
for all integers $n \ge 0$. Indeed, this is easy to confirm using the greedy algorithm. This minimal value is in fact realised, for it is a simple matter to check that
\[
R(F_m - 1) = 1 \qquad (m \ge 2).
\]
As for large values, Stockmeyer \cite{Sto2008} showed that
$R(n) \le \sqrt{n+1}$ for all $n \ge 0$, with equality if and only if $n = F_m^2 - 1$ for some $m \ge 2$.

Put
\[
\varphi = \frac{1 + \sqrt 5}{2}, \qquad
\lam = \frac{\log 2}{\log \varphi} \approx 1.44.
\]
We observe that $R(n)$ behaves somewhat erratically (see Figure \ref{OriginalFigure}). 
\begin{figure}[!htb]
  \centering
\includegraphics[width=8cm]{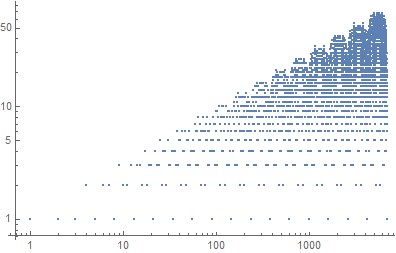}
  \caption{$R(n)$ against $n$ for $n=0,1,\ldots,6765$}
  \label{OriginalFigure}
\end{figure}

On the other hand, the summatory function behaves rather nicely, as seen in \cite{CS2021}.

\begin{thm}[Chow--Slattery, 2021]
\label{CS}
Define, for integers $H \ge 0$, 
    \[
    A(H) =\sum_{n=0}^H R(n).
    \]
    Then, for $H \ge 1$,  \[
    A(H) \asymp H^{\lambda}.
    \]
\end{thm}
Thus $R(n)$ has, on average, order of magnitude roughly $n^{0.44}$, which is slightly below Stockmeyer's upper bound.
Using similar methods, the present authors established an analogous result for the second moment in \cite{CJ2024}.

\begin{thm}[Chow--Jones, 2024]
\label{CJ}
Define, for integers $H \ge 0$, 
    \[
    V(H) =\sum_{n=0}^H R(n)^2.
    \]
    Then, for $H\ge 1$, \[
    V(H) \asymp H^{\log \lambda_1/\log \varphi},
    \]
    where  $\lam_1 \approx 2.48$ is  the greatest root of the polynomial 
\[
x^3 - 2 x^2 - 2 x + 2.
\]
\end{thm}

\noindent
This exponent is slightly larger than what is computed by combining Theorem~\ref{CS} with the Cauchy--Schwarz inequality.

Sanna then generalised Theorem~\ref{CJ} to all moments, using automata.

\begin{thm} [Sanna, 2025]
\label{Sanna}
For integers $k \ge 1$ and $H \ge 0$, define 
\[
M_k(H) = \sum_{n=0}^HR(n)^k.
\]
Then there exists $e_k > 0$ such that for $H \ge 1$,
\[
M_k(H) \asymp_k H^{(\log e_k)/\log \varphi}.
\]
Further,
\[
e_k^{1/k} \to \sqrt{\varphi}
\qquad (k \to \infty).
\]
\end{thm}

The asymptotic periodicity phenomenon visible in Figure \ref{OriginalFigure} has also been explained recently, in work of Kempton \cite{Kem}. In other recent work, Coons--Kristensen--Laursen \cite{CKL2025} considered the function counting partitions into Fibonacci numbers that are not required to be distinct, establishing an asymptotic formula for its logarithm.

In \cite{CJ2024}, it was also asked whether the `short average' of $R$ behaves as one would expect from Theorem \ref{CS}, specifically, whether
\[
\sum_{X < n \le X + \sqrt X} R(n) \asymp X^{\lam - 1/2}.
\]
Our main result is a negative answer to this question.
\begin{thm}
\label{MainThm}
Let $\alp \in (0,1)$, and let $m \ge 3$ be an integer. Then
\begin{equation}\label{Result}
\sum_{F_m < n \le F_m+F_m^\alp} R(n) \asymp_\alp (\log F_m) F_m^{\alp \lam}.
\end{equation}
\end{thm}

It is easy to check that 
\[
\alp \lam < \lam - 1 + \alp.
\]
These sums starting at $F_m$ correspond to certain local minima in Figure~\ref{Sgraph}.

\begin{figure}[!htb] 
\centering
\includegraphics[width=8cm]{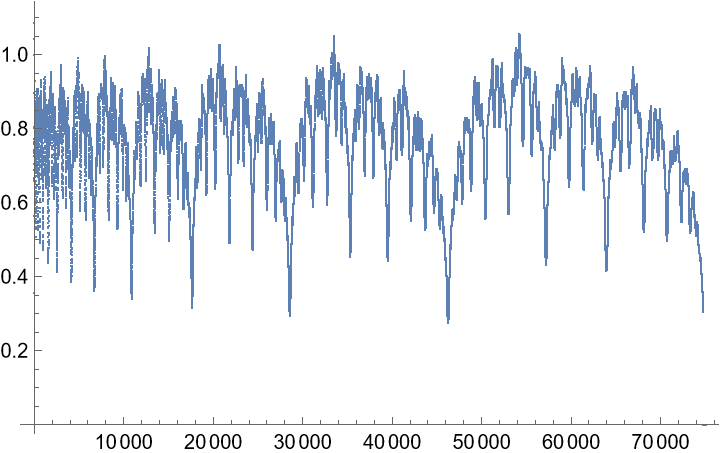}
\caption{$S(X)/X^{\lam-1/2}$ against $X$ for $1 \le X \le 74751$}
\label{Sgraph}
\end{figure}

\subsection*{Methods}

The key checkpoint is to find a nice expression for $R(F_m + t)$ involving $R(t)$. To achieve this, we begin with a combinatorial identity requiring an analysis of the underlying diophantine equation. We then iterate the identity to obtain our formula.

With this at hand, we sum over $t$ to compare
\[
\sum_{t=1}^H R(F_m + t) \quad
 \text{to}
 \quad
A(H),
\]
where $H = \lfloor F_m^\alp \rfloor$. Although these two sums have the same length, the former is philosophically a short average, whereas the latter is a long average that was understood in the work of Chow and Slattery \cite{CS2021}.

The comparison of short averages to long averages has played a substantial role in multiplicative number theory recently, following deep work of Matom\"aki and Radziwi\l \l~\cite{MR2016}. Our comparison principle is more specific and elementary, but the idea is likely applicable to other recurrence sequences.

\subsection*{Notation}

We adopt the Vinogradov and Bachmann--Landau notations, as we now describe. If $f$ and $g$ are complex-valued functions, then we write $f \ll g$ or $g \gg f$ or $f = O(g)$ if $|f| \le C|g|$ pointwise, for some constant $C>0$.

\subsection*{Funding} 

Sam Chow is supported by a grant of use under OpenAI's ChatGPT for Academic Researchers programme.
Owen Jones is supported by the Warwick Mathematics Institute Centre for Doctoral Training, and gratefully acknowledges funding from the Heilbronn Institute for Mathematical Research and the Engineering and Physical Sciences Research Council under the ``Additional Funding Programme for Mathematical Sciences'' (Grant number EP/V521917/1).


\subsection*{Acknowledgement and AI declaration}

We thank Gaurav Toshniwal for sharing an AI-generated version of the argument, and for allowing us free use of it. We took the main result and prompted GPT-6 Astra to produce its own derivation. We human-edited the latter into the proof here, and wrote the introduction ourselves.

\subsection*{Rights}

For the purpose of open access, the authors have applied a Creative Commons Attribution (CC-BY) licence to any Author Accepted Manuscript version arising from this submission.

\section{Proof of Theorem \ref{MainThm}}

The following identity can be deduced from \cite[Lemma 3.1]{CS2021} or \cite[Theorem 4]{Rob1996}, but it is more natural to show it directly.

\begin{lem} For integers $j \ge 5$ and $t \in [0, F_{j-3})$,
\[
R(F_j + t) = R(t) + R(F_{j-2} + t).
\]
\end{lem}

\begin{proof}
We count the number of solutions to
\begin{equation}\label{Fjt}
\sum_{i=1}^s x_i= F_j+t
\end{equation}
where $x_1 <\cdots <x_s$ are Fibonacci numbers and $s \ge 1$.
By the identity
\[
F_2+\ldots+F_{m-2}=F_m-2  \qquad (m \ge 3)
\]
and the fact that 
\[
F_j+t < F_j+F_{j-3} < F_{j+1},
\]
we must have $x_s \in \{F_{j-1},F_j\}$. The solutions for which $x_s=F_{j}$ are those satisfying
\[
\sum_{i=1}^{s-1}x_i=t
\]
with $x_1<\dots<x_{s-1}<F_j$, of which there are $R(t)$. On the other hand, substituting $x_s=F_{j-1}$ into \eqref{Fjt} yields
\[
\sum_{i=1}^{s-1}x_i=F_{j-2}+t<F_{j-1},
\]
where $x_1<\dots<x_{s-1}<F_j$, so this case contributes $R(F_{j-2}+t)$ many solutions.
\end{proof}

Iterating this furnishes the following.

\begin{lem}
Let $j\ge 5$ be a positive integer and $\ell$ be an even integer with $0 \le \ell \le j-3$. Then for every integer $t \in [0,F_{j-\ell-1})$ we have
    \[
    R(F_j+t)=\frac{\ell}{2}R(t)+R(F_{j-\ell}+t).
    \]
\end{lem}

\begin{proof}
We prove this by induction on $\ell$. For $\ell=0$ this is trivial. Suppose the statement holds for some $\ell\le j-5$. We wish to establish it with $\ell + 2$ in place of $\ell$.

Let $t \in [0,F_{j-\ell-3})$. Then, by the previous lemma, we have 
    \[
    R(F_{j-\ell}+t)=R(t)+R(F_{j-\ell-2}+t)
    \]
    and so
    \[
    R(F_j+t)=\frac{\ell+2}{2}R(t)+R(F_{j-\ell-2}+t).
    \]
This closes the induction and completes the proof.
\end{proof}

We can rewrite this as follows.

\begin{lem}
    For any integers $m \ge k \ge 3$ and $t \in [0,F_{k-1})$ with $m \equiv k \mmod 2$,
    \[
    R(F_m+t)=\frac{m-k}{2}R(t)+R(F_{k}+t).
    \]
    
\end{lem}

We can now sum over $t$ in the above to discover, for integers $H$ with $1 \le H <F_{k-1}$, that
\begin{equation}
\label{FmkSums}
\sum_{t=1}^H R(F_m+t)= \frac{(m-k) (A(H)-1)}{2}+\sum_{t=1}^H R(F_{k}+t),
\end{equation}
where 
\begin{equation}
\label{error}
0 \le \sum_{t=1}^H R(F_{k}+t) \le A(F_{k}+H) \ll (F_{k}+H)^{\lam} \ll F_k^\lam.
\end{equation}

We next apply \eqref{FmkSums} to prove Theorem \ref{MainThm}. Recall that 
\[
F_j \asymp \varphi^j \qquad(j \ge 1),
\]
and fix $\alpha \in (0,1)$. It is enough to show \eqref{Result} for $m$ large. In this case, let $H=\lfloor F_m^\alpha \rfloor \asymp \varphi^{\alpha m} $ and $k$ be the smallest integer such that 
\[
3 \le k \le m, \qquad m \equiv k \mmod 2 \qquad \text{ and } \qquad F_{k-1}>H.
\] Note that such a $k$ exists since
$
F_{m-3} \asymp \varphi^{m-3}
$
exceeds $H$
as $m$ is large. Now
$
F_{k-3} \le H<F_{k-1},
$
so 
\[
\varphi^{\alp m} \asymp H \asymp F_k \asymp \varphi^k.
\]
Taking logarithms in base $\varphi$, we find 
\[
m\alpha = k+O(1)
\]
and hence $m-k \asymp m$. Combining this with \eqref{FmkSums} and \eqref{error} yields
\[
\sum_{t=1}^H R(F_m+t)\asymp  m H^{\lambda} \asymp (\log F_m) F_m^{\alp\lam}.
\]

\end{document}